\documentclass[11pt]{amsart}
\usepackage{amsmath}
\usepackage{amssymb}
\usepackage{amsmath}
\usepackage{amsfonts}
\usepackage{latexsym}
\usepackage{color}
\usepackage{graphicx}

\newcommand{\be}{\begin{equation}}
\newcommand{\en}{\end{equation}}
\newcommand{\bea}{\begin{eqnarray}}
\newcommand{\ena}{\end{eqnarray}}
\newcommand{\beano}{\begin{eqnarray*}}
\newcommand{\enano}{\end{eqnarray*}}
\newcommand{\bee}{\begin{enumerate}}
\newcommand{\ene}{\end{enumerate}}

\newcommand{\mc}{\mathcal}

\newcommand{\A}{\mathfrak{A}}
\newcommand{\Ao}{\mathfrak{A}_0}

\newcommand{\D}{{\mc D}}

\newtheorem{defn}{Definition}[section]
\newtheorem{thm}[defn]{Theorem}
\newtheorem{prop}[defn]{Proposition}
\newtheorem{lemma}[defn]{Lemma}
\newtheorem{cor}[defn]{Corollary}
\newtheorem{example}[defn]{Example}
\newtheorem{rem}[defn]{Remark}

\def\x{\relax\ifmmode {\mbox{*}}\else*\fi}
\newcommand{\beex}{\begin{example}$\!\!${\bf }$\;$\rm }
\newcommand{\enex}{ \end{example}}
\newcommand{\berem}{\begin{rem}$\!\!${\bf }$\;$\rm }
\newcommand{\enrem}{ \end{rem}}
\newcommand{\bedefi}{\begin{defn}$\!\!${\bf }$\;$\rm }
\newcommand{\findefi}{\end{defn}}
\def\H{{\mathcal H}}

\newcommand{\ip}[2]{\langle {#1}|{#2}\rangle}

\begin{document}
\title[Extensions of representable positive linear functionals]
{Extensions of representable positive linear \\ functionals  to unitized quasi *-algebras}

\author[G. Bellomonte]{Giorgia Bellomonte}
\address{Dipartimento di Matematica e Informatica,
Universit\`a degli Studi di Palermo, Via Archirafi 34, I-90123 Palermo, Italy}

\email{bellomonte@math.unipa.it}

\subjclass{46K10, 46K70, 47L60}
\keywords{Representable linear functionals, unitized quasi *-algebras.}

\begin{abstract}
It is known that, under certain conditions, *-representability and extensibility to the unitized *-algebra of a positive linear functional, defined on a *-algebra without unit, are equivalent. In this paper, a new condition for an analogous result is given for the case of a hermitian linear functional defined on a quasi *-algebra $(\A,\Ao)$ without unit.
\end{abstract}

\maketitle

\section{Introduction}

The study of linear positive functionals on *-algebras and their *-representabi-lity has been undertaken since the $1970$s.
In \cite{Pow}, Powers  proved that, if $\omega$
is a positive linear functional on  a *-algebra $\A$ with unit, it is
possible to construct a closed strongly cyclic *-representation
$\pi$ of $\A$ on a Hilbert space $\H$, generalizing, in this way, the well known GNS construction for C$^*$-algebras to nontopological *-algebras.  But a positive
linear functional on a
*-algebra without unit does not extends automatically to a positive linear
functional on the unitization.

In the $1980$s, several authors (see \cite{Palmer} and references therein), improving \cite{Rickart},  showed that
a positive linear functional on a *-algebra $\A$ without unit can be extended to a positive linear functional $\omega^e$ on its unitization $\A^e=\A\oplus\mathbb{C}$ if, and only if, it is \emph{hermitian}, i.e., $\omega(a^*)=\overline{\omega(a)}$, $\forall a \in \A$ and \emph{Hilbert bounded}, i.e. there exists $\delta>0$ such
that $|\omega(a)|^2\leq\delta\omega(a^*a),$ for every $a \in\A$.
Hence the result of Powers applies to the extension $\omega^e$, and then the linear positive functional $\omega$ is *-representable, too.\\

Contrary to the case of *-algebras, a positive linear functional defined on a \emph{quasi} *-algebra $(\A,\Ao)$ with unit is \emph{not} automatically representable. For this reason, if $(\A,\Ao)$ has no unit, the notion of {\em extensibility}  has to be modified: a positive linear functional $\omega$ will be called \emph{extensible}  if it is the  restriction to $(\A\oplus\{0\},\Ao\oplus\{0\})$ of some \emph{representable} positive linear
functional $\omega^e$ defined on the unitization $(\A^e,\Ao^e)$.\\

Representability of a positive linear functional on a quasi *-algebra is  a topical problem, see e.g. \cite{FTT,ABT}, and it is of a certain interest in some application see, for instance,  \cite{BTT} and references therein.
   In this paper, starting from a  hermitian linear functional $\omega$ defined on a quasi *-algebra $(\A,\Ao)$ without unit,
    we study under what conditions it is possible to extend $\omega$ to a representable linear  functional,  defined on a quasi *-algebra with
unit, possibly the whole unitization of $(\A,\Ao)$.  Section 2 is devoted to recall the main definitions and
results about the subject. In Section 3, once defined the extension $\omega^e((x,\eta)):=\omega(x)+\eta\gamma$, $\gamma\in\mathbb{R}^+$, to the unitization of $(\A,\Ao)$, we give a new condition on $\omega$
to make the extension $\omega^e$ to the unitized quasi *-algebra $(\A^e,\Ao^e)$ (or to a quasi *-algebra with unit
contained in $(\A^e,\Ao^e)$) representable; this new condition is quite natural, in fact we prove that it reduces to Hilbert boundedness on $\Ao$; moreover if $(\A,\Ao)$ has a unit, then this condition is automatically fulfilled, hence, under this condition, representability of $\omega$ and its extensibility are equivalent.

\section{Definitions and known results}\label{defns}
Let $\D$ be a dense subspace of Hilbert space $\H$. We
denote by $\mathcal{L}^\dagger(\D,\H)$ the set of all (closable) linear operators $X$ such that
$D(X) = \D$, $D(X^*)\supseteq \D$. The set $\mathcal{L}^\dagger(\D,\H)$ is a partial *-algebra (see \cite{ait_book}),  i.e. it is a complex vector space with respect to the usual sum $X_1 + X_2$ and the scalar
multiplication $\lambda X$, with a conjugate linear involution  $X \to X^\dagger = X^*\upharpoonright \D$
and a distributive partial multiplication: the (weak) partial
multiplication $X_1 \Box X_2 = X_1^{\dagger *}X_2$, defined whenever $X_2$ is a weak right
multiplier of $X_1$ (equivalently, $X_1$ is a weak left multiplier of $X_2$), that is,
iff $X_2\D\subset D(X_1^{\dagger *})$ and $X_1^\dagger \D \subset  D(X_2^*)$ (we write $X_2 \in R^w(X_1)$ or $X_1\in
L^w(X_2)$). We denote with
$\mathcal{L}^\dagger(\D)$ the subspace of  $\mathcal{L}^\dagger(\D,\H)$ all closable operators $A$ in
$\H$ such that $A\D\subseteq\D$, $A^\dagger\D\subseteq\D$. This space with
this involution and the multiplication $\Box$,
is a
*-algebra and $A_1\Box A_2 \xi= A_1(A_2\xi)$ for every
$\xi\in\D$ \cite{schmud}.\\

A \emph{quasi *-algebra} is a couple $(\A,\Ao)$, where $\A$ is a vector space with
involution $^*$, $\Ao$ is a *-algebra and a vector subspace of $\A$ and $\A$ is an $\Ao$-
bimodule whose module operations and involution extend those of $\Ao$.
 A quasi*-algebra is said {\em topological} (resp. \emph{locally convex})  if $\A$ is a topological (resp. locally convex) vector space with topology $\tau$, with continuous involution $^*$,
$\Ao$ is a dense vector subspace of $\A[\tau]$ and the module operations of $\A$ and involution are separately, not necessarily jointly, continuous
bilinear mappings of $\Ao\times\A$ and $\A\times\Ao$ to $\A$,
where $\Ao$ carries the induced topology of $\A$ \cite{ ait_book,
schmud}.\\

 A
\emph{*-representation of a quasi *-algebra} $(\A,\A_0)$ is a
*-homomorphism $\pi$ of $\A$ into
$\mathcal{L}^\dag(\mathcal{D}_\pi,\mathcal{H}_\pi)$,
for some pair $(\mathcal{D}_\pi,\mathcal{H}_\pi)$ with $\mathcal{D}_\pi$ a dense subspace
 of Hilbert space $\mathcal{H}_\pi$ \cite{ait_book}, i.e. a linear map
$\pi:\A\to\mathcal{L}^\dag(\mathcal{D}_\pi,\mathcal{H}_\pi)$ such
that:\begin{itemize}
  \item[i)] $\pi(x^*)=\pi(x)^\dag$, \hspace{0,2cm}$\forall x\in \A$,
  \item[ii)]  if $x\in\A$, $a\in\A_0$, then $\pi(x)$ is a weak left multiplier of $\pi(a)$ and $\pi(x)\Box\,\pi(a)=\pi(xa)$,
  \item[iii)] $\pi(a)\in\mathcal{L}^\dag(\mathcal{D}_\pi)$, $\forall a \in\Ao$.
\end{itemize}
If $(\A,\A_0)$ has unit $e$, we assume $\pi(e)=\mathbb{I}.$\\

If $\pi$ is a *-representation of $(\A,\Ao)$ in $\D_\pi$, then $\D_\pi$ can be endowed with a
locally convex topology $t_\pi$ defined by
the family of seminorms
$x\in\D_\pi\mapsto\|\pi(x)\xi\|$,  $x\in\A$.
Put $$\D(\widetilde\pi) = \bigcap_{x\in \A}
D(\overline{\pi(x)})$$
where $\overline{\pi(x)}$ denotes the closure of the operator $\pi(x)$. Then one defines
$\widetilde\pi(x):= \overline{\pi(x)}\upharpoonright\D(\widetilde\pi)$, $x\in\A$. The *-representation $\widetilde{\pi}$ is said the \emph{closure} of $\pi$.
If $\D(\pi)=\D(\widetilde{\pi})$ and $\pi=\widetilde{\pi}$, the representation is
said to be \emph{closed}.
If $\pi$ is a *-representation of $(\A,\Ao)$ in $\D_\pi$, define $\D_ {\pi^*}=\bigcap_{x\in\A}D(\pi(x)^*)$, then the adjoint representation $\pi^*$ is such that $\pi^*(x)=\pi(x^*)^*\upharpoonright\D_{\pi^*}$, $\forall x\in\A$. If $\pi=\pi^*$, then the *-representation $\pi$ is said to be \emph{self-adjoint}.

A *-representation $\pi$ of $(\A,\Ao)$ is called
\begin{itemize}
\item \emph{cyclic} if there exists a vector $\xi\in\D_\pi$ such that $\{\pi(a)\xi; a \in \Ao\}$ is dense in $\H_\pi$;
\item \emph{strongly cyclic} if there exists a vector $\xi\in\D_\pi$ such that $\{\pi(a)\xi; a \in \Ao\}$ is dense in $\D_\pi$ with respect to the topology $t_\pi$. \end{itemize}

In \cite{ct_studia08} Trapani
gave three conditions for a linear
functional $\omega$, defined on a  quasi
*-algebra (possibly without unit), to be representable.

 A linear functional $\omega$ defined on the quasi *-algebra $(\A,\Ao)$, is said to be \emph{representable}   if
\begin{itemize}
\item[(L.1)]$\omega(a^*a) \geq 0, \quad \forall a \in \A_0$;
    \item[(L.2)]$\omega(b^*x^* a)= \overline{\omega(a^*xb)}, \quad
\forall a,b \in \A_0, \forall x \in \A$;
    \item[(L.3)]$\forall x \in \A$, there exists $\gamma_x >0$ such
    that $$|\omega(x^*a)| \leq \gamma_x \omega(a^*a)^{1/2}, \quad \forall a
    \in \A_0.$$
\end{itemize}Let us call $\lambda_\omega(a), a \in\A_0$,
the coset containing $a$ in the space
$\lambda_\omega(\A_0):=\A_0/N_\omega$, where $N_\omega=\{a\in\A_0 :
\omega(a^*a) = 0\}$ is a left-ideal of $\A_0$.  The quotient $\lambda_\omega(\A_0)$ is a
pre-Hilbert space with inner product
$\ip{\lambda_\omega(a)}{\lambda_\omega(b)} = \omega(b^*a),\,\, a,
b\in\A_0$. Let $\mathcal{H}_\omega$ be the completion of
$\lambda_\omega(\A_0)$. As proved in \cite[Theorem 3.5
]{ct_studia08}, even if the quasi
*-algebra $(\A,\A_0)$ has not unit conditions (L.1), (L.2) and
(L.3) allow a generalized GNS construction for the quasi
*-algebra $(\A,\A_0)$: there exists a
*-representation $\pi_\omega$ of  $(\A, \A_0)$  such that for any $x\in\A$   $$\omega(b^*xa)=\ip{\pi_\omega(x)\lambda_\omega(a)}{\lambda_\omega(b)},$$ $\forall a,b\in\Ao.$ \\ If $(\A,\Ao)$ has a unit $e$, $\omega$ admits a closed *-representation $\widetilde{\pi}_\omega$ of $(\A,\Ao)$, with strongly cyclic vector $\xi_\omega=\lambda_\omega(e)$ by  completing the domain.
\\

If, in particular, $\A=\A_0$ is a *-algebra without unit, every positive linear functional $\omega$ on $\A_0$
is *-\emph{representable} in the sense that there exist a subspace
$\D(\pi_\omega)\subset\H_{\pi_\omega}$, a closed *-representation
$\pi_\omega:\A\to\mathcal{L}^\dagger(\D(\pi_\omega))$  such that $\omega(b^*xa)=\ip{\pi_\omega(x)\lambda_\omega(a)}{\lambda_\omega(b)}, \forall a,b,x \in\A.$
    In this case, properties (L.1)-(L.3) are automatically  satisfied by $\omega$. Moreover, if $\A$ possesses the unit, then $\omega$ is cyclically *-representable, i.e.
     there exists
 a vector
$\xi_\omega\in\D(\pi_\omega)$ such that $\omega(x)
=\ip{\pi_\omega(x)\xi_\omega}{\xi_\omega},$ $\forall x \in\A$ and
$\xi_\omega$ is strongly cyclic. The representation $\pi_\omega$ is unique up to unitary equivalence.

\begin{defn} Let $(\A, \A_0)$ be a quasi *-algebra without unit, its {\em unitization}
is the couple $(\A^e, \A_0^e)$ where $\A^e=\A\oplus\mathbb{C}$ and
$\A_0^e=\A_0\oplus\mathbb{C}$. The partial multiplication of $\A$ is
extended to $\A^e$ as follows
$$(x,\lambda)\cdot(a,\mu)=(xa+\mu x+\lambda a,\lambda\mu),\qquad
x\in\A, \,a\in\A_0,\,\, \lambda,\mu\in\mathbb{C}$$ and the adjoint
of an element $(x,\eta)$ is
$$(x,\eta)^*=(x^*,\overline{\eta}),\qquad\,x\in\A,\,\, \eta\in\mathbb{C}.$$Together with the usual operations of the direct sum space, the previous partial multiplication and the involution make $(\A^e,\Ao^e)$ into a quasi *-algebra with unit $(0,1)\in\Ao^e$.
\end{defn}

\begin{defn}
 A  linear functional $\omega$, defined on a quasi
*-algebra $(\A,\Ao)$ without unit is called {\em extensible} if $\omega$ is equal to the restriction to \\ $(\A\oplus\{0\},\Ao\oplus\{0\})$ of some representable linear
functional $\omega^e$ defined on the unitization $(\A^e,\Ao^e)$.
\end{defn}

Before going forth, let us show that, in general, a representable positive linear functional $\omega$ defined on a quasi*-algebra $ (\A,\Ao)$ without unit is not automatically extensible. In some cases it is not even possible to extend its restriction $\omega\upharpoonright\Ao$ to the unitization of the *-algebra $\Ao$, as the in following example.

\begin{example}\label{integral}
Let $(\A,\Ao)=(L^1(\mathbb{R})\bigcap L^2(\mathbb{R}),C_0(\mathbb{R}))$. It is a quasi *-algebra without unit. Let us consider the functional $$\omega(f)=\int_{\mathbb{R}}f(x)dx,\quad f\in L^1(\mathbb{R})\bigcap L^2(\mathbb{R}).$$

It is easy to see that $\omega$ enjoys properties (L.1)-(L.3). Indeed,
as for (L.3), $\forall f\in L^1(\mathbb{R})\bigcap L^2(\mathbb{R})$ and $\forall g\in C_0(\mathbb{R})$, $$|\omega(f^*g)|=\Big{|}\int_\mathbb{R} \overline{f(x)}g(x)dx\Big{|}\leq\|f\|_2\|g\|_2=\gamma_f\omega(g^*g)^{1/2}$$

Hence $\omega$ is representable on $L^1(\mathbb{R})\bigcap L^2(\mathbb{R})$.

For $\omega$ to be extensible to the unitization $(L^1(\mathbb{R})\bigcap L^2(\mathbb{R}))\oplus\mathbb{C}$
it is necessary that it is extensible to the unitization of $C_0(\mathbb{R})$, but this does not happen;
indeed, it is not Hilbert bounded; were it so, one should have $$\|\omega\|_H=\sup\{ |\omega(g)|^2; g \in C_0(\mathbb{\mathbb{R}}), \omega(g^*g)=1\}<\infty, $$
but if we take the functions $f_n(x)\in C_0(\mathbb{R})$ with

$$f_n(x)=\left\{
        \begin{array}{ll}
          nx+n^2+1 & \hbox{if $x\in[-n-1/n,-n[$;} \\
          1 & \hbox{if $x\in[-n,n]$;} \\
          -nx+n^2+1, & \hbox{if $x\in]n,n+1/n]$;} \\
          0 & \hbox{elsewhere,}
        \end{array}
      \right.
$$
we get
$\omega(f_n)=\int_\mathbb{R} f_n(x)dx\geq 2n$, $|\omega(f_n)|^2=\Big{|}\int_\mathbb{R} f_n(x)dx\Big{|}\geq 4n^2$ and
$\omega(f^*_nf_n)=\int_\mathbb{R} |f_n(x)|^2dx\leq\int_{-2n}^{2n} dx=4n$, hence the quotient  $\frac{|\omega(g)|^2}{\omega(g^*g)}>n$.
It is thus clear that we cannot extend $\omega$ to the unitized quasi *-algebra.
\end{example}

 We will maintain the notation of \cite{ct_studia08}.

\section{Representable extensions to a unitized quasi *-algebra}\label{direct extension of omega}

  Let $(\A, \Ao)$ be a quasi *-algebra without unit and  $\omega$ be a positive linear functional  defined on $(\A, \Ao)$. Let $\omega$ be hermitian (i.e., $\omega(x^*)=\overline{\omega(x)}$, $\forall x \in \A$), then (L.2) is automatically satisfied.
Assume, in addition, that  \mbox{$\omega_0=\omega\upharpoonright\Ao$} is also \emph{Hilbert
bounded} on the
*-algebra $\Ao$; i.e., the following property holds
\begin{itemize}
\item[(HB)]there exists $\delta>0$ such
that $$|\omega_0(a)|^2\leq\delta\omega_0(a^*a),\qquad \forall a \in\Ao.$$
\end{itemize}

Then  $\omega_0$ is positive (thus property (L.1) is
satisfied).\\ The Hilbert bound $\|\omega_0\|_H$ is defined by
$$\|\omega_0\|_H = \sup\{ |\omega_0(a)|^2; a\in\Ao, \omega_0(a^*a)\leq1\};$$
if $\Ao$ has a unit $e$, then as a simple consequence of the Cauchy-Schwarz inequality,
$|\omega_0(a)|^2 \leq \omega_0(e )\omega_0(a^*a),$ $\forall a \in\Ao$;
i.e. $\omega_0$ is Hilbert bounded and $\|\omega_0\|_H = \omega_0(e)$.

  If (HB) holds true, then it is known that $\omega_0$ can be extended to $\Ao^e$; indeed, the functional $\omega^e_0$ defined on the unitized *-algebra  $\A^e_0$ by
\begin{equation}\label{extens to the unitiz *-alg}
\omega^e_0((a,\lambda))=\omega_0(a)+\lambda\gamma_0,\qquad(a,\lambda)\in\Ao^e
\end{equation} with $\gamma_0\geq\|\omega_0\|_H$ is positive. If $\xi$ is the cyclic vector
of the GNS representation of $\omega^e_0$, then
\begin{equation}\label{inequality}
 |\omega_0(a)|^2
\leq\|\xi\|^2 \omega_0(a^*a),\qquad\forall a\in\A_0
 \end{equation}
with $\|\xi\|^2=\|\omega^e_0\|_H=\omega^e_0((0,1))=\gamma_0\geq\|\omega_0\|_H,$ then
\eqref{extens to the unitiz *-alg} can be rewritten as $$\omega^e_0((a,\lambda))=\omega_0(a)+\lambda\|\xi\|^2,\qquad(a,\lambda)\in\Ao^e.$$

\begin{rem} If $\Ao$ is a C*-algebra without unit and $\vartheta$ is a positive functional on $\Ao$, then
 $\vartheta$ is automatically Hilbert bounded and therefore extensible to $\Ao^e$ (see, e.g. \cite[Corollary 2.3.13]{BR}).\end{rem}

 Assume now that $\omega$   satisfies property
 (L.3) too. Hence it is representable. Now we
extend $\omega$ to $\A^e$, for every $(x,\eta)\in\A^e$ by defining\footnote{In order to lighten the notation we do not stress the dependence of $\omega^e$ on $\gamma$; to be more precise to each $\gamma\geq\|\xi\|^2$ is associated a different extension of $\omega$.}
\begin{equation}\label{defn omega^e with cyclic vector}
\omega^e((x,\eta)):=\omega(x)+\eta\gamma
\end{equation}
with $\gamma\geq\|\xi\|^2$.\\

We want to present some conditions for the representability of
$\omega^e$. Condition (L.1) holds
also for $\omega^e$ on $\A_0^e$; indeed it suffices to observe that
$\forall\, (a,\lambda)\in\A_0^e$, by using \eqref{inequality} and hermiticity
\begin{align*}
 \omega^e((a^*,\overline{\lambda})\cdot(a,\lambda)) &=
 \omega(a^*a)+\lambda\omega(a^*)+\overline{\lambda}\omega(a)+|\lambda|^2\gamma \\
   &\geq \omega(a^*a)-2|\lambda||\omega(a)|+|\lambda|^2\gamma\\
   &\geq \omega(a^*a)-2|\lambda| \gamma^{1/2}\omega(a^*a)^{1/2}+|\lambda|^2\gamma\\
   &= (\omega(a^*a)^{1/2}-|\lambda|\gamma^{1/2} )^2\geq0.
\end{align*}

Now,  the hermiticity of $\omega$ implies that $\omega^e$ is
hermitian too. Indeed,
$$\omega^e((x^*,\bar{\eta}))=\omega(x^*)+\bar{\eta}\gamma=\overline{\omega(x)+\eta\gamma}=\overline{\omega^e((x,\eta))},\hspace{0,2cm}\forall(x,\eta)\in\A^e. $$

As for property (L.3), even if $\omega$ is representable, in general  it  is not  true for $\omega^e$,  as the following example shows.
\begin{example}\label{ex: not (L.3)}
Let us consider the quasi *-algebra $(\A,\Ao)$, where $\A=\ell_2$ is, as usual, the space of square summable complex sequences and
$\Ao$ the *-algebra of all finite sequences, i.e.:
$$\Ao:=\{\textbf{a}=(a_n): \, \exists N\in \mathbb{N}, \,\,\,a_n=0,\,\forall
n>N\},$$ with componentwise multiplication, and let us  define a  linear functional $\omega$ on $\ell_2$ by
 \begin{equation}\label{finite sums}\omega((x_n))=\sum_{n=1}^{+\infty}
x_n, \qquad (x_n)\in \ell_2.\end{equation} As we showed in \cite{BProc}, $\omega$ is representable both as functional defined on the whole quasi *-algebra and if we restrict it on the *-algebra $\Ao$. However, its extension defined as in (\ref{defn omega^e with cyclic vector}), does not enjoy property (L.3). Indeed, if we consider $\textbf{x}=(x_n)$ with $x_n=1/n$, $\forall n  \in\mathbb{N}$, we have, for $\mu\neq0$ and for all $a\in\Ao$
$$|\omega^e((\textbf{x},\eta)^*\cdot (\textbf{a},\mu))|^2=|\omega(\textbf{x}^*\textbf{a})+\mu\omega(\textbf{x}^*)+\overline{\eta}\omega(\textbf{a})+\overline{\eta}\mu\gamma|^2
=|K+\mu\sum_{n=1}^{+\infty}
1/n|^2=\infty$$ since $K=\omega(\textbf{x}^*\textbf{a})+\overline{\eta}\omega(\textbf{a})+\overline{\eta}\mu\gamma$ contains only finite sums.\end{example}

 It is then necessary to introduce a further condition in order that (L.3) holds for $\omega^e$. To do this, let us define $$p_x(\omega):=\sup\{|\omega(x^*a)|^2:a\in\Ao, \omega(a^*a)=1\}.$$

The functional  $x\mapsto p_x(\omega)^{1/2}$ is a seminorm: indeed,
\begin{enumerate}
  \item[i)] $p_x(\omega)\geq 0 \,\,\forall x\in\A$, then $p_x(\omega)^{1/2}\geq0$,
  \item[ii)]  if $x = 0$, then $p_x(\omega)^{1/2} = 0$,
  \item[iii)] $p_{\lambda x}(\omega)^{1/2}= |\lambda| p_x(\omega)^{1/2}$ $\forall x\in\A,\lambda\in\mathbb{C}$,
  \item[iv)] $p_{x+y}(\omega)^{1/2}\leq p_{x}(\omega)^{1/2}+p_{y}(\omega)^{1/2}$.
\end{enumerate}

Now we are ready to introduce a necessary and sufficient condition for  $\omega$ to be extensible:
 \begin{flushleft}
 (EHB) there exists $\zeta>0$ such that
$$ |\omega(x^*)|\leq \zeta p_x(\omega)^{1/2},\quad\forall x\in\A.$$
\end{flushleft}

\begin{rem}
If $(\mathfrak{B},\mathfrak{B}_0)$ is a quasi *-algebra with unit $e\in\mathfrak{B}_0$, then (EHB) is automatically fulfilled by a positive linear functional $\vartheta$ defined on it; indeed, for every $x\in\mathfrak{B}$ $$|\vartheta(x^*)|^2=|\vartheta(x^*e)|^2\leq \sup\{|\vartheta(x^*a)|^2:a\in\mathfrak{B}_0, \vartheta(a^*a)=1\}= p_x(\vartheta)^{1/2}.$$ Hence condition (EHB) is  necessary for the extensibility of a positive linear functional $\omega$ defined on a quasi *-algebra without unit, to the unitized quasi *-algebra; moreover, (EHB) is also a sufficient condition for the extensibility of $\omega$ as Proposition \ref{prop: estendability} below shows.\end{rem}

Condition (EHB) is a natural one; indeed, the following lemma holds.

\begin{lemma}\label{(EHB) implies HB}
Let us consider a  linear functional $\omega$ on a quasi *-algebra $(\A,\Ao)$ and suppose that it enjoys property (L.1). If $\omega$ enjoys (EHB), then it  enjoys property (HB) too. The two conditions are equivalent on $\Ao$.
\end{lemma}
\begin{proof}
It suffices to prove that properties (HB) and (EHB) coincide on $\Ao$. Let $b\in\Ao$, then $p_b(\omega)=\sup\{|\omega(b^*a)|^2;\,\omega\in\Ao\,\,\omega(a^*a)=1\}$; by the Cauchy-Schwarz  inequality $|\omega(b^*a)|^2\leq\omega(b^*b)\omega(a^*a)$ we get that $p_b(\omega)\leq\omega(b^*b)$. More precisely, the equality holds: indeed, if $\omega(b^*b)=0$ then $p_b(\omega)=0$; other-wise take $a=b/\omega(b^*b)^{1/2}\in\Ao$; then, $\omega(a^*a)=1$ and $p_b(\omega)=\omega(b^*b).\hfill\qedhere$
\end{proof}
According the previous lemma, (EHB) extends (HB) to the whole quasi *-algebra, whence the acronym (EHB).\\

The main result of this paper then holds.
\begin{prop}\label{prop: estendability}
Let $\omega$ be a hermitian linear functional defined on a quasi *-algebra $(\A,\Ao)$ without unit, such that property (L.3)  holds.
If (EHB) holds too, then the extension $\omega^e$ to the unitized quasi *-algebra $(\A^e,\Ao^e)$ is cyclically representable, i.e. the functional $\omega$ is extensible to the unitized quasi *-algebra.
\end{prop}
\begin{proof} Let us consider the functional $\omega^e$ defined in (\ref{defn omega^e with cyclic vector}). Property (HB) holds for $\omega_0=\omega\upharpoonright\Ao$  by Lemma \ref{(EHB) implies HB}.
We have already observed that properties (L.1) and (L.2) hold for $\omega^e$, because of (HB) and the hermiticity of $\omega$.  It remains to verify that, if (L.3) and (EHB) hold for $\omega$, property (L.3) holds for $\omega^e$ too.
Let us estimate the squared modulus of $\omega^e((x^*,\overline{\eta})\cdot(a,\mu))$
\begin{align*}
 &   |\omega^e((x^*,\overline{\eta})\cdot(a,\mu))|^2=\\
   & = |\omega(x^*a)|^2+|\eta|^2|\omega(a)|^2+|\mu|^2|\omega(x^*)|^2+|\eta|^2|\mu|^2\gamma^2 +2\Re(|\eta|^2\overline{\mu}\gamma\omega(a))+ \\
   &\quad + 2\Re\Big{(}[\omega(x^*)+\overline{\eta}\gamma]\mu\overline{\omega(x^*a)}+\overline{\eta}\omega(a)\overline{\omega(x^*a)}+ \overline{\eta}\overline{\mu}\omega(x)\omega(a)+\eta|\mu|^2\gamma\omega(x^*)\Big{)}\\
&  \leq (\gamma_x^2+|\eta|^2\gamma)\omega(a^*a)+ \left(\frac{\zeta^2p_x(\omega) }{\gamma}+|\eta|^2\gamma\right)|\mu|^2\gamma +2|\eta|^2\gamma\Re(\overline{\mu}\omega(a))+\\
  & \quad+ 2\Big{[}\big{(}\zeta p_x(\omega)^{1/2}+|\eta|\gamma\big{)}|\mu|\gamma_x\omega(a^*a)^{1/2}+|\eta|\gamma^{1/2}\gamma_x\omega(a^*a)+\\
  &\quad+\big{(}\gamma^{1/2}\omega(a^*a)^{1/2}+|\mu|\gamma\big{)}|\eta||\mu|\zeta p_x(\omega)^{1/2}\Big{]}
\\&\leq(\gamma_x+|\eta|\gamma^{1/2})^2\omega(a^*a)+
\left(\frac{\zeta p_x(\omega)^{1/2} }{\gamma^{1/2}}+|\eta|\gamma^{1/2}\right)^2|\mu|^2\gamma +2|\eta|^2\gamma\Re(\overline{\mu}\omega(a))\\
&\quad+2\Big{[}\big{(}\zeta p_x(\omega)^{1/2}+|\eta|\gamma\big{)}\gamma_x+\gamma^{1/2}|\eta|\zeta p_x(\omega)^{1/2}+|\eta|^2\gamma^{3/2}\Big{]}|\mu|\omega(a^*a)^{1/2}
\\
&\leq \left[(\gamma_x+|\eta|\gamma^{1/2})\omega(a^*a)^{1/2}+ \left(\frac{\zeta p_x(\omega)^{1/2} }{\gamma^{1/2}}+|\eta|\gamma^{1/2}\right)|\mu|\gamma^{1/2}\right]^2\\&\quad+2|\eta|^2\gamma\Re(\overline{\mu}\omega(a)).
\end{align*}
 Put $$K_{(x,\eta)}:=\max\left\{\gamma_x+|\eta|\gamma^{1/2},\frac{\zeta p_x(\omega)^{1/2}}{\gamma^{1/2}}+|\eta|\gamma^{1/2}\right\},$$ and $$\Gamma_{(x,\eta)}:=\max\{2 (K_{(x,\eta)})^2,|\eta|^2\gamma\}.$$ Then
 \begin{align*}
 & |\omega^e((x^*,\overline{\eta})\cdot(a,\mu))|  \leq(K_{(x,\eta)})^2\Big{[}\omega(a^*a)^{1/2}+|\mu|\gamma^{1/2}\Big{]}^2+2|\eta|^2\gamma\Re(\overline{\mu}\omega(a))\\
 &\leq 2 (K_{(x,\eta)})^2\Big{[}\omega(a^*a)+|\mu|^2\gamma\Big{]}+2|\eta|^2\gamma\Re(\overline{\mu}\omega(a))\\
&\leq\Gamma_{(x,\eta)}\Big{[}\omega(a^*a)+|\mu|^2\gamma+2\Re(\overline{\mu}\omega(a))\Big{]}=\Gamma_{(x,\eta)}\omega^e((a^*,\overline{\mu})\cdot(a,\mu)).
\end{align*}
The functional $\omega^e$ is then representable on the whole $(\A^e,\A_0^e)$.
\end{proof}

\begin{cor}\label{cor:equivalence}Let $\omega$ be a hermitian linear functional defined on a quasi *-algebra $(\A,\Ao)$ without unit, such that property (EHB)  holds. Then the representability of $\omega$ is equivalent to its extensibility.
\end{cor}

\begin{example}
A very simple example of the situation described in Proposition \ref{prop: estendability} is obtained by considering the quasi *-algebra $(\A,C_0(\mathbb{R}))$, with $\A=\{f\in C(\mathbb{R}):\int_\mathbb{R}|f(x)|e^{-x}dx<\infty\},$ $C_0(\mathbb{R})$ the *-algebra of continuous functions with compact support and  the linear functional defined on it $\omega(f)=f(0)$.
 It is easy to prove that $\omega$ is hermitian and  that (L.1) and (L.3) hold. The functional $\omega$ enjoys also property (EHB) on $\A$. Indeed,
\begin{align*}
  p_x(\omega)&= \sup\{|\omega(f^*g)|^2:g\in\Ao, \omega(g^*g)=1\} \\
   &=   \sup\{|\overline{f(0)}g(0)|^2:g\in\Ao, |g(0)|^2=1\} =|\overline{f(0)}|^2.
\end{align*}
 It is then possible to extend $\omega$ to a positive linear functional on the unitization of $(\A,\Ao)$:
 $\omega^e((f,\eta)):=\omega(f)+\eta\gamma=f(0)+\eta\gamma$
with $\gamma\geq\|\xi\|^2\geq\|\omega_0\|_H=1$.
 \end{example}

We are now able to extend a well-known theorem for positive linear functionals on *-algebras without unit, to positive linear functionals defined on quasi *-algebras without unit, by summarizing the remarks and results given until now.

\begin{thm}\label{thm: implications and equiv main}
Let $\omega$ be a  linear functional on $(\A,\Ao)$ for which (L.1) holds and consider the following statements.
\begin{itemize}
\item[$(i)$] $\omega^e:\A^e\to \mathbb{C}$ defined by        $$\omega^e ((x, \lambda )) = \omega(x) + \gamma\lambda\,, \qquad (x,\lambda)\in\A^e$$
      is a representable linear functional on $\A^e$, for every $\gamma\geq\|\omega^e_0\|_H$;
\item[$(ii)$]  there exists a *-representation $\pi$ of $\A$  and a vector\footnote{In general, it is \emph{not} a cyclic vector for $\pi$, but only  for $\pi_{\omega^e}$ a *-representation of $\omega^e$.} $\xi\in \D(\pi_{\omega^e})$ such that $\omega(x) = \ip{\pi(x)\xi}{\xi},\,\forall x\in \A$;
\item[$(iii)$] $\omega$ is  hermitian, enjoys (HB) and for every $x\in\A$ there exists a vector $\zeta_\omega(x)\in \H_\omega$ such that
$$\qquad\qquad\qquad\qquad\quad\omega(x^*a) = \ip{\lambda_\omega(a)}{\zeta_\omega(x)},\qquad  \forall a\in\Ao.$$
Then the following implications hold $$(i)\Rightarrow (ii)\Rightarrow (iii).$$
Moreover, if there exists $\zeta>0$ such that
$$ |\omega(x^*)|\leq \zeta p_x(\omega)^{1/2},\quad\forall x\in\A$$ then the three statements $(i)$, $(ii)$ and $(iii)$ are equivalent and also equivalent to the following ones
\item[$(iv)$] $\omega$ is extensible;
\item[$(v)$] $\omega$ is hermitian and enjoys  property (L.3);
\item[$(vi)$] $\omega$ is hermitian and representable.
\end{itemize}
\end{thm}
\begin{proof}
$(i)\Rightarrow (ii)$: it suffices to consider the GNS representation of the unitized quasi *-algebra $(\A^e,\Ao^e)$ constructed (as in \cite{ct_studia08}) from
$\omega^e$, restrict it to $(\A\oplus\{0\},\Ao\oplus\{0\})$, define $\pi(x):=\pi_{\omega^e}(x,0)$, $\forall x\in\A$ and take $\xi=\xi_{\omega^e}=\lambda_{\omega^e}((0,1))\in\D(\pi_{\omega^e})$ (but , in general, $\xi\notin\D(\pi_{\omega^e}\upharpoonright\A\oplus\{0\})$).  \\
$(ii)\Rightarrow (iii)$: if $\pi$ is a *-representation as described in $(ii)$, then $\omega$ is hermitian, and $\forall a\in\Ao$
$$|\omega(a)|^2 = |\ip{\pi(a)\xi}{\xi}|^2\leq\|\xi\|^2\|\pi(a)\xi\|^2=\|\xi\|^2\ip{\pi(a^*a)\xi}{\xi}=\|\xi\|^2\omega(a^*a)$$
further $\forall x\in\A$
$$\omega(x^*a)=\ip{\pi(x^*a)\xi}{\xi}=\ip{\pi(a)\xi}{\pi(x)\xi}=\ip{\lambda_\omega(a)}{\zeta_\omega(x)},\quad \forall a\in\Ao$$ remembering that $\xi=\lambda_{\omega^e_0}((0,1))$.\\
Let from now on (EHB) hold for $\omega$.\\
$(iii)\Rightarrow (iv)$:  it follows easily that (L.3) holds for $\omega$. Then, by Proposition
 \ref{prop: estendability} we get the result.\\
 $(iv)\Rightarrow (v)$: assume that $\omega'$ is a representable linear functional which extends $\omega$ to $\A^e$.
Then, $\omega$ is hermitian and (L.3) holds for $\omega$ too.\\
$(v)\Rightarrow (i)$: by hypotheses $\|\omega^e_0\|_H<\infty$ and let
$\gamma\geq\|\omega^e_0\|_H$.
The proof of  (L.1) for $\omega^e$ is standard, (L.2) follows from hermiticity of $\omega$ and (L.3) for $\omega^e$ follows from Proposition \ref{prop: estendability}.\\
$(iv)\Rightarrow (vi)$: see Corollary \ref{cor:equivalence}.
\end{proof}

\begin{rem}\label{cyclicity, HB and self adj + boundedn of pi}
Let (HB) hold for $\omega_0=\omega\upharpoonright\Ao$.
Then, the linear functional $$F(\lambda_{\omega_0}(a))=\omega_0(a)$$
 is Riesz representable, i.e. there exists $\xi_{\omega_0}\in\lambda_{\omega_0}(\Ao)$
such that $$\omega_0(a)=\ip{\lambda_{\omega_0}(a)}{\xi_{\omega_0}}.$$
Now, $\forall b\in\Ao$, $$\omega_0(ba)=\ip{\lambda_{\omega_0}(ba)}{\xi_{\omega_0}}=\ip{\pi_{\omega_0}(b)\lambda_{\omega_0}(a)}{\xi_{\omega_0}}$$
and, on the other hand $$\omega_0(ba)=\ip{\lambda_{\omega_0}(a)}{\lambda_{\omega_0}(b^*)},$$
hence $\pi_{\omega_0}(b)^*\xi_{\omega_0}=\lambda_{\omega_0}(b^*)$; we can conclude that
$\xi_{\omega_0}\in\D(\pi^*_{\omega_0})$.\\ If $\pi_{\omega_0}$ is self-adjoint, then $\lambda_{\omega_0}(a)=\pi_{\omega_0}(a)\xi_{\omega_0}$, $a\in\Ao$. In particular, if  $\pi_{\omega_0}$ is bounded (i.e., if $\omega_0$ is admissible), then $\xi_{\omega_0}$ is a cyclic vector for $\pi_{\omega_0}$ and $\omega_0(a)=\ip{\pi_{\omega_0}(a)\xi_{\omega_0}}{\xi_{\omega_0}}$; hence $\omega_0$ is strongly cyclically representable.
Conversely, with the same argument as in $(ii)\Rightarrow (iii)$ in Theorem \ref{thm: implications and equiv main} it is clear that, if $\omega$ is cyclically representable, then it enjoys (HB). We cannot conclude it from the Theorem \ref{thm: implications and equiv main} directly, because the representation in $(ii)$ is not cyclic, in the sense that the vector $\xi$ is not cyclic for $\pi$.
\end{rem}

There are cases in which property (EHB) does not hold, as the following examples show.

\begin{example}
Let us consider both the quasi *-algebra and the linear functional of Example \ref{ex: not (L.3)}.  There,  we already observed that  $\omega$ is representable as a functional defined on the whole quasi *-algebra (hence (L.3) is satisfied see also \cite{BProc}) as well as restriction to the *-algebra $\Ao$.  The functional $\omega$ is hermitian, but it does not enjoy property (EHB).

Let us calculate, for any $\textbf{a}\in\Ao$ and for a fixed $\textbf{x}\in\ell_2$,
$$
\frac{|\omega(\textbf{x}^*\textbf{a})|^2}{\omega(\textbf{a}^*\textbf{a})}=\frac{|\sum_{n=1}^{N}
\overline{x_n}a_n|^2}{\sum_{n=1}^{N}
|a_n|^2}\leq\frac{\sum_{n=1}^{N}
|x_n|^2\sum_{n=1}^{N}
|a_n|^2}{\sum_{n=1}^{N}
|a_n|^2}=\sum_{n=1}^{N}
|x_n|^2<\infty.$$
Hence $p_\textbf{x}(\omega)\leq\|\textbf{x}\|_2^2$, but property (EHB) does not  hold. Indeed, it suffices to consider a sequence
such that $|\omega(\textbf{x}^*)|>\zeta\|\textbf{x}\|_2\geq \zeta p_\textbf{x}(\omega)^{1/2}$, for every $\zeta>0$ i.e. such that
$$\left|\sum_{n=1}^{\infty}
\overline{x_n}\right|=\left|\sum_{n=1}^{\infty}
x_n\right|>\zeta\left(\sum_{n=1}^{\infty}
|x_n|^2\right)^{1/2},\quad\forall\zeta>0.$$
An easy example of this situation is obtained by taking $\textbf{x}=\left(\frac{1}{n}\right)_{n\in\mathbb{N}}$.
Hence, the task of extending $\omega$ to the unitization of $\ell_2$ is hopeless, here.
\end{example}

\begin{example}
Let us consider the same functional and the quasi *-algebra without unit of Example \ref{integral}.
Let us consider the functions $f_n\in L^1(\mathbb{R})\bigcap L^2(\mathbb{R})$ defined as $$f_n(x)=\left\{
                                                                                                \begin{array}{ll}
                                                                                                  1/\sqrt{2n}, & \hbox{$x\in[-n,n]$;} \\
                                                                                                  0, & \hbox{otherwise.}
                                                                                                \end{array}
                                                                                              \right.
$$ They have unit norm, but
$$ |\omega(f_n^*)|=\left|\int_{-n}^n1/\sqrt{2n}\,\,dx\right|=\frac{1}{\sqrt{2n}}2n=\sqrt{2n}\to \infty,\mbox{ as } n\to\infty.$$
Then, property (EHB) does not hold. We already noted in Example \ref{integral} that this functional is not extensible. \end{example}

If $\omega$ is hermitian and properties (HB) and (L.3) hold but condition (EHB) is not fulfilled on the  whole $\A$, it is possible to extend $\omega$ to a quasi *-algebra with unit  $(\D(\omega^e),\A_0^e)$ with  $\D(\omega^e)\subset\A^e$. For this, let us  restrict the linear functional $\omega^e$ defined in \eqref{defn omega^e with cyclic vector} to the vector space
\begin{align*}
   \D(\omega^e)\hspace{-0,3cm}&:= \hspace{-0,2cm}\{(x,\eta)\in\A^e:\,\exists\, \delta_{(x,\eta)}
>0 \mbox{ s.t. } \\
   & |\omega^e((x,\eta)^*\cdot(a,\mu))|\leq\delta_{(x,\eta)}\omega^e((a,\mu)^*\cdot(a,\mu))^{1/2},\,\, \forall(a,\mu)\in\A_0^e\}.
\end{align*}

We notice that
  $\D(\omega^e)$ contains
$\Ao^e$ and $\omega^e_0$ satisfies  (HB) since $\omega^e_0=\omega^e\upharpoonright\Ao^e$ is a positive linear functional on a
*-algebra (with unit) and hence it satisfies the Cauchy-Schwarz inequality.
 As for the algebraic structure of
$\D(\omega^e)$, it is in general a nonempty
  subspace of $\A^e$. But there are
situations where something more can be said. For instance, let us
assume that the following condition of {\em admissibility} holds:
\begin{flushleft}
 (ADM)  $\forall(b,\nu)\in\Ao^e$ there exists $\delta_{(b,\nu)}>0$ such that
$$\omega^e_0((a,\mu)^*\cdot(b,\nu)^*\cdot(b,\nu)\cdot(a,\mu))\leq\delta_{(b,\nu)}\omega^e_0((a,\mu)^*\cdot(a,\mu)),\,\, \forall(a,\mu)\in\Ao^e.$$
\end{flushleft}
\begin{rem} If $\Ao$ is a C*-algebra, then (ADM) is automatically satisfied.
\end{rem}

Let us suppose that $\D(\omega^e)$ is stable under involution.
\begin{prop}\label{D quasi *-alg} Let $\omega$ be a hermitian linear functional on a
quasi *-algebra without unit, satisfying  (HB) and (L.3). If $\omega^e_0$
satisfies (ADM) and  $\D(\omega^e)$ is stable under involution, then
 $(\D(\omega^e),\A_0^e)$ is a quasi *-algebra with unit.\end{prop}
\begin{proof}The vector space $\D(\omega^e)$ contains $\A_0^e$ and is *-stable.
 Besides, by (ADM),
$(b,\nu)\cdot(x,\eta)\in\D(\omega^e)$, with
$(x,\eta)\in\D(\omega^e)$ and $(b,\nu)\in\A_0^e$.\end{proof}

Properties (L.1)-(L.3)  are hence satisfied by $\omega^e$ on
$\D(\omega^e)$ and $\omega^e_\D:=\omega^e\upharpoonright\D(\omega^e)$ is representable.\\
Let us maintain the notation of Section \ref{defns};
for every $(x,\eta)\in\D(\omega^e)$, the functional
 $(x,\eta)^{\omega^e_\D}:\lambda_{\omega^e_\D}(\A_0^e)\to\mathbb{C}$ with
$$(x,\eta)^{\omega^e_\D}(\lambda_{\omega^e_\D}(a,\mu)):=\omega^e_\D((x,\eta)^*\cdot(a,\mu)),$$
is linear and bounded.  By applying the Riesz lemma, we can prove
the following.
\begin{thm}\label{HB+L.3}Let $\omega$ be a hermitian linear functional on a
quasi *-algebra $(\A,\Ao)$ without unit, satisfying properties (HB) and (L.3).
If $(\D(\omega^e), \Ao^e)$ is a quasi
*-algebra with unit then $\omega^e_\D$ is
representable.
\end{thm}
By Proposition \ref{D quasi *-alg}, then we have the following

\begin{cor}\label{cor HB+L.3}Let $\omega$ be a hermitian linear functional on a
quasi *-algebra $(\A,\Ao)$ without unit, satisfying properties (HB) and (L.3).
If $\omega^e_0$ satisfies also (ADM)  and  $\D(\omega^e)$ is stable under involution, then
 $\omega^e_\D$ is representable on the quasi *-algebra with unit $(\D(\omega^e),\A_0^e)$.
\end{cor}

\begin{rem} If $\omega^e_0$ satisfies (ADM), then the GNS-representation
restricted to $\A_0^e$, constructed from $\omega^e_\D$, is
bounded.\end{rem}

 \subsection*{Acknowledgment} The author
wishes to thank Prof. A. Inoue and Prof. C. Trapani for their fruitful comments.


\begin{thebibliography}{1}
\bibitem{ait_book} J-P. Antoine, A. Inoue, and C. Trapani, \textit{Partial *-Algebras and Their Operator
   Realizations}, (Mathematics and Its Applications: Vol. 553).   Kluwer Academic Publishers, 2002.
\bibitem{ABT} J-P. Antoine, G. Bellomonte and C. Trapani,  \textit{Fully representable and *-semisimple topological partial *-algebras}  Studia Math., no. 2 \textbf{208} (2012), 167--194 DOI: 10.4064/sm208-2-4.
\bibitem{BTT} F. Bagarello, C. Trapani and S. Triolo,  \textit{Representable states  on quasi-local quasi *-algebras}, J. Math. Phys., \textbf{52} (2011), no. 1, 013510, 11 pp.
\bibitem{BProc}  G. Bellomonte,  \textit{Quasi *-algebras arising from extensions
of positive linear functionals}. Rend. Circ. Mat. Palermo (2) Suppl. \textbf{82} (2010),  235--249.
\bibitem{BR} O. Bratteli, and  D.W. Robinson, \textit{ Operator Algebras and Quantum Statistical Mechanics, Vol. 1} (Texts and monographs in physics). Springer-Verlag 1979.
\bibitem{FTT} M. Fragoulopoulou, C. Trapani and S. Triolo, \textit{Locally convex quasi *-algebras with sufficiently
many *-representations}, J. Math. Anal. Appl. \textbf{388} (2012), 1180--1193.
\bibitem{Palmer} T.W. Palmer, \textit{Banach Algebras and the General Theory of *-Algebras, Vol. II *-algebras.} (Encyclopedia of mathematics and its applications: Vol. 79)  Cambridge University Press, 2001.
\bibitem{Pow} R.T. Powers, \textit{Self-adjoint algebras of unbounded operators}. Comm. Math. Phys., \textbf{21} (1971) 85--124.
\bibitem{Rickart} C.E. Rickart, \textit{General theory of Banach algebras}, (The University Series in Higher Mathematics ).  Van Nostrand, 1960.
\bibitem{schmud} K. Schm\"udgen, \textit{Unbounded Operator Algebras and Representation Theory}, (Operator Theory: Advances and Applications, Vol.37). Birkh\"auser - Verlag,  1990.
\bibitem{ct_studia08} C. Trapani,  \textit{*-Representations, seminorms and structure properties of normed quasi
*-algebras}. Studia Math., \textbf{186} (2008) 47--75.
\end{thebibliography}
\end{document}